\documentclass[a4paper,11pt]{article}
\usepackage{amsfonts,amsthm,amsmath,amssymb,graphicx}
\usepackage{color}

\title{Normal sequences with given limits of multiple ergodic averages}

\date{}

\author{Lingmin Liao\\
\normalsize LAMA UMR 8050, CNRS, Universit\'e Paris-Est Cr\'eteil,\\
\normalsize 61 Avenue du G\'en\'eral de Gaulle, 94010 Cr\'eteil Cedex, France\\
\normalsize e-mail: lingmin.liao@u-pec.fr\\
Micha\l \ Rams \\
\normalsize Institute of Mathematics, Polish Academy of Sciences\\
\normalsize ul. \'Sniadeckich 8, 00-656 Warszawa, Poland\\
\normalsize e-mail: rams@impan.pl}

\theoremstyle{plain}
\newtheorem{lem}{Lemma}[section]
\newtheorem{prop}[lem]{Proposition}
\newtheorem{thm}[lem]{Theorem}
\newtheorem{cor}[lem]{Corollary}

\theoremstyle{definition}

\theoremstyle{remark}

\numberwithin{equation}{section}

\newcommand{\N}{\mathbb N}

\renewcommand{\epsilon}{\varepsilon}

\begin{document}

\maketitle

\def\thefootnote{}
\footnote{2010 {\it Mathematics
Subject Classification}: Primary 28A80, Secondary 11K16, 37B40}
\def\thefootnote{\arabic{footnote}}
\setcounter{footnote}{0}

\begin{abstract}
We are interested in the set of normal sequences in the space $\{0,1\}^\N$ with a given frequency of the pattern $11$ in the positions $k, 2k$. The topological entropy of such sets is determined. 
\end{abstract}

\section{Introduction and statement of results}


Recently, Fan, Liao, Ma \cite{FLM}, and Kifer \cite{K12} proposed to calculate the topological entropy spectrum of level sets of multiple ergodic averages. Here, the topological entropy means Bowen's topological entropy (in the sense of \cite{B73}, see the definition in Section \ref{sec:pre}) which can be defined for any subset, not necessarily invariant or closed. 

Let $\Sigma=\{0,1\}^{\mathbb{N}}$.  Among other questions, Fan, Liao, Ma \cite{FLM} asked for the topological entropy of 
\[
A_{\alpha}:=\Big\{(\omega_k)_1^{\infty}\in \Sigma : \lim_{n\to\infty} \frac{1}{n}\sum_{k=1}^{n} \omega_k\omega_{2k}=\alpha \Big\} \qquad (\alpha\in [0,1]).
\]
As a first step to solve the question, they also suggested to study a subset of $A_0$:
\[  A:=\Big\{(\omega_k)_1^{\infty}\in \Sigma : \omega_k\omega_{2k}=0 \quad \text{for all } k\geq 1\Big\}.\]
 The topological entropy of $A$ was later given by Kenyon, Peres and Solomyak \cite{KPS12}.
\begin{thm}[Kenyon-Peres-Solomyak]\label{thm:KPS} 
We have
\[
h_{\rm top} (A)=-\log (1-p)=0.562399...,  \]
where $p\in [0,1]$ is the unique solution of $$p^2=(1-p)^3.$$
\end{thm}
Enlightened by the idea of \cite{KPS12}, the question about the topological entropy of $A_\alpha$ was finally answered by Peres and Solomyak \cite{PeSo12}, and then in higher generality by Fan, Schmeling and Wu \cite{FSW16}.
\begin{thm}[Peres-Solomyak, Fan-Schmeling-Wu]\label{thm:PS}
For any $\alpha\in[0,1]$, we have
\[
h_{\rm top}( A_{\alpha})=-\log (1-p)- \frac{\alpha}{2} \log \frac{q(1-p)} {p(1-q)},
\] where $(p,q)\in [0,1]^2$ is the unique solution of the system 
 \[
 \left\{ 
 \begin{array}{ll}
\displaystyle p^2(1-q)=(1-p)^3,\\
 \displaystyle {2pq} =\alpha({2+p-q}). 
\end{array}\right.
\]
In particular, $h_{\rm top}(  A_{0})= h_{\rm top}(  A)$.
\end{thm}

Another, interesting, related set is
\[
B:=\Big\{(\omega_k)_1^{\infty}\in \Sigma : \omega_k=\omega_{2k} \quad \text{for all } k\geq 1\Big\}.
\]

A sequence $ { \omega} \in \{0,1\}^\N$ is said to be simple normal if the frequency of the digit $0$ in the sequence is $1/2$. It is said to be normal if for all $n\in \mathbb{N}$, each word in { $\{0,1\}^{n}$} has frequency $1/2^n$. We denote the set of normal sequences by $\mathcal{N}$.

We are interested in the intersection of  $\mathcal{N}$  with the set $A_{\alpha}$ of given frequency of the pattern $11$ in $\omega_k\omega_{2k}$. For the usual ergodic (Birkhoff) averages the normal sequences all belong to one set in the multifractal decomposition -- the situation for multiple ergodic averages turns out to be very different. 

Our results are as follows:
\begin{thm} \label{thm:norm}
For $\alpha \leq 1/2$ we have
\[
h_{\rm top} (\mathcal{N} \cap A_{\alpha}) = {1 \over 2} \log 2+ {1 \over 2} H(2\alpha),
\] 
where $H(t)=-t\log t -(1-t)\log (1-t)$. For $\alpha>1/2$ the set $\mathcal{N} \cap A_{\alpha}$ is empty.

Further, 
\[ h_{\rm top}  (\mathcal{N} \cap  A)=h_{\rm top}  (\mathcal{N} \cap A_{0}) ={1\over 2} \log 2.\]
Moreover, ${\mathcal N}\cap B \subset A_{1/2}$ and 

\[
h_{\rm top}  (\mathcal{N} \cap  B)=h_{\rm top}  (\mathcal{N} \cap A_{1/2}) = h_{\rm top}  (B) ={1\over 2} \log 2.
\]

\end{thm}
The last statement of Theorem \ref{thm:norm} was recently proved, in higher generality, by Aistleitner, Becher and Carton \cite{ABC}.

\smallskip
Let us now define the set of sequences with prescribed frequency of $0$'s and $1$'s:
\[
E_\theta:=\{ { \omega} \in \Sigma: \lim_{n\to\infty} {\omega_1({ \omega})+\cdots +\omega_n({\omega}) \over n}=\theta\}.
\]
In particular, $E_{1/2}$ is the set of simple normal sequences.
\begin{thm} \label{thm:freq}
We have
\[
 h_{\rm top}  (E_\theta \cap A_{\alpha})=\big(1-{\theta \over 2}\big)H\Big({2\theta- \alpha \over 2-\theta}\Big) +{\theta \over 2}H\Big({\theta- \alpha \over \theta}\Big)
\]
for $\alpha \leq \theta \leq (2+\alpha)/3$, otherwise $E_\theta \cap A_{\alpha}=\emptyset$.
Further,
\[  h_{\rm top}  ( E_\theta \cap A) = h_{\rm top}  ( E_\theta \cap A_0)={2-\theta\over 2} H\Big({2\theta \over 2-\theta}\Big).\]
\end{thm}

Note that
\[  h_{\rm top}  (E_{1/2} \cap A)={3\over 4} H\Big({2 \over 3}\Big)>h_{\rm top}  (\mathcal{N} \cap  A).\]

Applying the results of \cite{PeSo12}, we have the following corollary. 
\begin{cor}\label{cor}
The equality
\[
h_{\rm top}  (E_\theta \cap A_{\alpha})=h_{\rm top}  ( A_{\alpha})
\]
holds if and only if $\alpha, \theta$ satisfy the relation
\begin{align}\label{eq:max}
(2\theta-\alpha)^2(\theta-\alpha)(2-\theta)=\theta(2-3\theta+\alpha)^3.
\end{align}
In particular, when 
\[\theta={2\over 3}\left(1+\left({2 \over 23}\right)^{2/3} \sqrt[3]{3\sqrt{69}-23}- \left({2 \over 23}\right)^{2/3} \sqrt[3]{3\sqrt{69}+23}\right)=0.354...,\]
i.e., the unique real solution of the equation $4 \theta^2 (2 - \theta) =(2 - 3 \theta)^3$, 
we have 
\[\dim_H E_\theta \cap A= \dim_H A.
\]
\end{cor}

We organize our paper as follows. In Section \ref{sec:pre}, we give some preliminaries. Section \ref{sec:3} is devoted to the proof of Theorem \ref{thm:norm}. The proofs of Theorem \ref{thm:freq} and Corollary \ref{cor} are given in Section \ref{sec:4}.

\bigskip
\section{Preliminary}\label{sec:pre}

\subsection{Bowen's topological entropy}
In 1973, Bowen \cite{B73} introduced a definition of topological entropy for any subset which is not necessarily invariant or closed. 
Though the original definition of Bowen's topological entropy is for any topological dynamical systems, we recall, for simplicity, the definition of Bowen's entropy for a topological dynamical system $(X,T)$ equipped with a metric $d$.  For $x\in X$, $n\in \mathbb{N}$, $n\geq 1$, denote by $B_n(x, \epsilon)$ the Bowen ball defined by
\[
B_n(x,\epsilon):=\left\{y\in X: d(T^kx, T^ky)<\epsilon, \ \forall k=0, \dots, n-1\right\}.
\]
For $E\subset X$, $s\geq 0$, $N\geq 1$ and $\epsilon>0$, set 
$$
     \mathcal{H}^s_N (E, \epsilon) = \inf \sum_i \exp(-s n_i),
$$
where the infimum is taken over all finite or countable families $\{B_{n_i}(x_i, \epsilon)\}$ such that $x_i\in X, n_i\geq N$ and $E \subset \bigcup_i B_{n_i}(x_i,\epsilon)$.
  The quantity $ \mathcal{H}^s_N (E, \epsilon)$ is non-decreasing as $N$ increases, so the following limit exists
$$
    \mathcal{H}^s (E, \epsilon)  = \lim_{N \to \infty}  \mathcal{H}^s_N (E, \epsilon).
$$
For the quantity $ \mathcal{H}^s (E, \epsilon)$ considered as a function of
$s$, there exists a critical value, which we denote by $h_{\rm
top} (E, \epsilon)$, such that
$$
      \mathcal{H}^s (E, \epsilon) =\left\{ \begin{array} {ll} +\infty, & s <
                        h_{\rm top} (E, \epsilon) \\ 0 , & s> h_{\rm
                        top} (E, \epsilon).  \end{array} \right.
$$
One can prove that the following limit exists
$$
          h_{\rm top} (E) = \lim_{\epsilon \to 0} h_{\rm top} (E,
          \epsilon).
$$
The quantity $h_{\rm top} (E)$ is called the {\em topological
entropy} of $E$.

We remark that for the symbolic dynamical system $(\Sigma, \sigma)$ where the space $\Sigma=\{0,1\}^{\mathbb{N}}$ is equipped with the usual metric defined by 
\[
\forall \omega, \tau \in \Sigma, \quad d(\omega, \tau):=2^{-\min\{n\geq 0: \ \omega_{n+1} \neq \tau_{n+1}\}},
\]
and $\sigma$ is the left shift defined by 
\[
\sigma : \omega_1\omega_2 \dots \mapsto \omega_2\omega_3 \dots,
\]
the Bowen ball $B_n(\omega, \epsilon)$ ($\epsilon<1$) is nothing but the cylinder $C_n(\omega_1, \cdots, \omega_n)$ of order $n$ defined by
\[
C_n(\omega_1, \cdots, \omega_n):=\{\tau\in \Sigma : \tau_1=\omega_1, \cdots, \tau_n=\omega_n\},
\] 
and Bowen's topological entropy $h_{\rm top} (E) $ and Hausdorff dimension $\dim_H(E)$ of a subset $E\subset X$ are different only with a constant:
\[
h_{\rm top} (E) = (\log 2)\cdot \dim_H(E).
\] 
We refer to Falconer's book \cite{F90} for the details on Hausdorff dimension.

\subsection{Billingsley's lemma for Bowen's entropy}
The Mass Distribution Principle (\cite[Principle 4.2]{F90}), or more generally, Billingsley's lemma (\cite{Bi61}) for the Hausdorff dimension has the following topological entropy version (\cite{MW08}).

Let $\mu$ be a Borel probability measure on $X$. The local entropy $h_{\mu}(x)$ at a point $x\in X$ is defined as 
\[
h_{\mu}(x)=\lim_{\epsilon \to 0}\liminf_{n\to\infty} {- \log \mu (B_n(x,\epsilon)) \over n}.
\]

\begin{thm}[Ma-Wen 2008]\label{lem:Bi}
Let $\mu$ be a Borel probability measure on $X$, $E\subset X$ be a Borel subset and $0<h<\infty$. Then
\begin{enumerate}
\item[i)] if $h_\mu(x) \leq h$ for all $x\in E$, then $h_{\rm top} (E) \leq h$,
\item[ii)] if $h_\mu(x) \geq h$ for all $x\in E$, and $\mu(E)>0$, then $h_{\rm top} (E) \geq h$.
\end{enumerate}
\end{thm}

We remark that in the symbolic dynamical system $(\Sigma, \sigma)$, the local entropy $h_\mu(\omega)$ at a point $\omega\in \Sigma$ is 
\[
h_\mu(\omega)=\liminf_{n\to\infty} {-\log \mu(C_n(\omega)) \over n},
\]
where $C_n(\omega)$ is the cylinder of order $n$ containing the point $\omega$.

\subsection{A lemma of elementary analysis}

The following lemma of Peres and Solomyak (\cite[Lemma 5]{PeSo12}) will be applied several times.
\begin{lem}[Peres and Solomyak] \label{lem:PS}
Suppose that $\{z_n\}$ is a bounded real sequence and there exists $c > 0$ such that
$|z_n-z_{n+m}|\leq c{{m \over n}}$  for all $m,n \in \mathbb{N}$.
If $z_{2^kn}\rightarrow \gamma$ as $k\to \infty$ for all $n\in \mathbb{N}$, then $z_n\to \gamma$.
\end{lem}

\subsection{A family of measures on $\{0,1\}^\mathbb{N}$}\label{subsec:measures}

For the lower bound estimations of the topological entropy, the following family of measures on $\{0,1\}^\mathbb{N}$ will be used.
Let $(p_0, p_1)$ be a probability vector, i.e., $p_0, p_1\geq0$ and $p_0+p_1=1$. Let 
\[
\begin{pmatrix}
   p_{00} & p_{01} \\
   p_{10} & p_{11} 
\end{pmatrix}
\]
be a transition matrix with all coefficients $p_{ij}\geq0$ and $p_{i0}+p_{i1}=1$ for $i=0,1$.
We also assume the following condition which will guarantee our measures to be non-trivial:
\begin{align}\label{cond:nontrivial}
|p_{01}-p_{11}| \neq 2. 
\end{align}

With the data $p_i, p_{ij}$ ($i,j\in \{0,1\}$), we define a Borel measure $\mu$ on the space $\Sigma= \{0,1\}^\mathbb{N}$ as follows
\begin{itemize}
\item[--] if $k$ is odd then $\omega_k=1$ with probability $p_1$,
\item[--] if $k$ is even and $\omega_{k/2}=1$ then $\omega_k=1$ with probability $p_{11}$,
\item[--] if $k$ is even and $\omega_{k/2}=0$ then $\omega_k=1$ with probability $p_{01}$,
\end{itemize}
with the events $\{\omega_k=1\}$ and $\{\omega_\ell =1\}$ independent except when $k/\ell$ is a power of 2. 
More precisely, 
we define the measure $\mu$ on any cylinder $C_n(\omega_1\cdots\omega_{n})$ of order $n\in\mathbb{N}$ by
\[
\mu(C_n(\omega_1\cdots\omega_{n}))=\prod_{k=1}^{\lceil n/2\rceil}p_{\omega_{2k-1}} \cdot \prod_{k=1}^{\lfloor n/2\rfloor} p_{\omega_{k}\omega_{2k}}.
\]
where $\lceil \cdot \rceil, {\lfloor \cdot\rfloor}$ denote the ceiling function and the integer part function. 
Then by Kolmogorov consistence theorem, $\mu $ is well defined on $\Sigma$. We remark that the measure $\mu$ depends on the given data $p_i, p_{ij}$ ($i,j\in \{0,1\}$). We will see late that by suitablely choosing these data, we can find suitable measures supported on the sets, using which we calculate the topological entropy. 

\smallskip
For $\omega\in \Sigma$ and $n\in \mathbb{N}$, set
\[
x_n(\omega) = \frac 2 n \sum_{k=n/2+1}^n \omega_k.
\]
The following Lemmas \ref{lem:LLN}, \ref{lem:wk-w2k} and \ref{lem:h-n} implied by the strong Law of Large Numbers will be useful.
\begin{lem}\label{lem:LLN}
For $\mu$-almost all  $\omega$ and for big enough $n$, 
\[
x_{2n}(\omega) = {\frac 1 2}p_1  + \frac {x_n(\omega)} 2 p_{10} + \frac {1-x_n(\omega)} 2 p_{01} + o(1).
\]
\end{lem}
\begin{proof}
Recall that, by the definition of the measure $\mu$, the events $\{\omega_k=1\}$ and $\{\omega_\ell =1\}$ are independent except when $k/\ell$ is a power of $2$.
 While in the average of $x_n(\omega)$ there is no different $k$'s with quotient bing a power of $2$. Thus for those $n/2+1 \leq k \leq n$, $\omega_k$ are independent. They are also identically distributed. Further, it is evident that the expectation and variance of these $\omega_k$ are finite. Thus by the classical strong Law of Large Numbers, 
\begin{align}\label{eq:x-n}
x_n(\omega)=\mathbb{E}_\mu(\omega_k)+o(1)=\mu (\omega_k=1)+o(1), \quad \mu-a.e.
\end{align}
where $n/2+1 \leq k \leq n$.  

Note that for $n+1\leq k \leq 2n$, we have
\begin{align*}
&\mu (\omega_k=1)\\
=&\mathbb{P}(\{k \text{ is odd} \}) \cdot \mu( \omega_k=1| \{k \text{ is odd} \}) +\mathbb{P}(\{k \text{ is even} \}) \cdot  \mu( \omega_k=1| \{k \text{ is even} \})\\
=&{1 \over 2} p_1 + {1 \over 2} \mu (\omega_m=1) \cdot \mu(\omega_k=1|k=2m,  \omega_m=1)\\
&+ {1 \over 2}  \mu (\omega_m=0) \cdot \mu(\omega_k=1|k=2m, \omega_m=0)\\
=&{1 \over 2} p_1 + {1 \over 2} \mu (\omega_m=1) \cdot p_{11}+ {1 \over 2}  \mu (\omega_m=0) \cdot p_{01},
\end{align*}
where $m=k/2$, so $n/2+1\leq m\leq n$.

Therefore, for $\mu$-almost all  $\omega$ and for big enough $n$, 
\[
x_{2n}(\omega) = {\frac 1 2}p_1  + \frac {x_n(\omega)} 2 p_{11} + \frac {1-x_n(\omega)} 2 p_{01} + o(1).
\]
\end{proof}

Applying Lemma \ref{lem:PS}, we can determine the $\mu$-almost sure limit of $x_n(\omega)$.
\begin{cor}\label{cor:lim-x-n}
For $\mu$-almost all sequence $\omega$,
\[
\lim_{n\to\infty} x_n(\omega) =\xi:= {p_1+p_{01} \over 2-p_{11}+p_{01}}.
\]
\end{cor}
\begin{proof}
Note that by the condition (\ref{cond:nontrivial}), $$\left|{p_{01}-p_{11} \over 2}\right|<1.$$
Thus by Lemma \ref{lem:LLN}, $\mu$-almost surely, as $k\to \infty$,
\[
x_{2^k n}(\omega) \to {p_1+p_{01} \over 2-p_{11}+p_{01}}.
\]
By Lemma \ref{lem:PS}, this implies that $\mu$-almost surely
\begin{equation} \label{eq:freq}
\lim_{n\to\infty} x_n(\omega) =  {p_1+p_{01} \over 2-p_{11}+p_{01}}.
\end{equation}
\end{proof}

\begin{prop}\label{prop:lim-wk}
For $\mu$-almost all sequence $\omega$,
\[
\lim_{n\to\infty} {1 \over n} \sum_{j=1}^n\omega_j =\xi= {p_1+p_{01} \over 2-p_{11}+p_{01}}.
\]
\end{prop}
\begin{proof}
By Corollary \ref{cor:lim-x-n}, for $\mu$-almost all sequence $\omega$, for all $n$,
\begin{align*}
\lim_{k\to\infty} {1 \over 2^kn} \sum_{j=1}^{2^kn}\omega_j=\lim_{k\to\infty} {1 \over 2^kn} \left(\sum_{i=1}^{k}2^{i-1}n \cdot x_{2^i n}(\omega) + \sum_{j=1}^n\omega_j\right)=\xi.
\end{align*}
Applying Lemma \ref{lem:PS}, we have 
\begin{align*}
\lim_{n\to\infty} {1 \over n} \sum_{j=1}^n\omega_j=\xi, \qquad \mu-a.e.
\end{align*}
\end{proof}

\begin{lem}\label{lem:wk-w2k}
We have for $(i,j)\in \{0,1\}^2$, for $\mu$-almost all  $\omega$ and for big enough $n$, 
\[
\frac 2 n \sum_{k=n/2+1}^n \mathbf{1}_{\{\omega_k \omega_{2k}=ij\}}(\omega) = \left(\frac 2 n \sum_{k=n/2+1}^n \mathbf{1}_{\{\omega_k =i\}}(\omega)\right) \cdot p_{ij}+o(1),
\]
and in particular,
\[
\frac 2 n \sum_{k=n/2+1}^n \omega_k \omega_{2k} = x_n(\omega)  \cdot p_{11}+o(1).
\]
\end{lem}
\begin{proof}
Note that for $n/2+1 \leq k \leq n$ the variables $\omega_k\omega_{2k}$ are independent and identically distributed. Thus for $(i,j)\in \{0,1\}^2$, by the classical strong Law of Large Numbers, for $\mu$-almost all $\omega\in \Sigma$, 
\[
\frac 2 n \sum_{k=n/2+1}^n \mathbf{1}_{\{\omega_k \omega_{2k}=ij\}}(\omega) = \mathbb{E}_\mu(\mathbf{1}_{\{\omega_k \omega_{2k}=ij\}})=\mu(\omega_k\omega_{2k}=ij)+o(1).
\]
Note that 
\[
\mu(\omega_k\omega_{2k}=ij)=\mu(\omega_k=i)\cdot \mu(\omega_k\omega_{2k}=ij|\omega_k=i)=\mu(\omega_k=i)\cdot p_{ij}.
\]
Hence, by \eqref{eq:x-n}, we complete the proof. 
\end{proof}

By Lemma \ref{lem:wk-w2k} and Corollary \ref{cor:lim-x-n}, we immediately obtain the following corollary. 
\begin{cor}\label{cor:wk-w2k} 
We have for $j\in \{0,1\}$, for $\mu$-almost all  $\omega$, 
\[
\lim_{n\to\infty}\frac 2 n \sum_{k=n/2+1}^n \mathbf{1}_{\{\omega_k \omega_{2k}=0j\}}(\omega) = (1-\xi) \cdot p_{0j},
\]
and
\[
\lim_{n\to\infty}\frac 2 n \sum_{k=n/2+1}^n \mathbf{1}_{\{\omega_k \omega_{2k}=1j\}}(\omega) = \xi \cdot p_{1j},
\]
and in particular, 
\[
\lim_{n\to\infty}\frac 2 n \sum_{k=n/2+1}^n \omega_k \omega_{2k} = \xi \cdot p_{11}.
\]
\end{cor}

\begin{prop}\label{prop:wk-w2k}
For $\mu$-almost all sequence $\omega$,
\[
\lim_{n\to\infty} {1 \over n} \sum_{j=1}^n\omega_j\omega_{2j} =\xi \cdot p_{11}.
\]
\end{prop}
\begin{proof}
The proof is the same as that of Proposition \ref{prop:lim-wk} by using Corollary \ref{cor:wk-w2k} and Lemma \ref{lem:PS}.
\end{proof}

\smallskip
For $n\in \mathbb{N}$ and $\omega\in \Sigma$, denote
\[
h_n(\omega) := \log \mu(C_{2n}(\omega)) - \log \mu(C_n(\omega)).
\]
\begin{lem}\label{lem:h-n}
For $\mu$-almost all sequence $\omega$ and for big enough $n$, we have
\begin{align*}
\frac 2 n h_n(\omega) = &p_0\log p_0+p_1\log p_1+ (1-x_n(\omega))( p_{00} \log p_{00}+ p_{01} \log p_{01}) \\
&+ x_n(\omega)( p_{10} \log  p_{10}+ p_{11} \log  p_{11})+ o(1).
\end{align*}
\end{lem}
\begin{proof}
Following \cite{PeSo12}, for positive integers $m<n$, we write $\omega_m^n$ for the word $\omega_m\omega_{m+1}\cdots \omega_n$. For $i,j\in \{0,1\}$ and $\omega\in \Sigma$, denote
\[
N_i(\omega_m^n)=\sharp \{m\leq k \leq n: \omega_k=i\},
\]
and 
\[
N_{ij}(\omega_m^n)=\sharp \{m\leq k \leq n: \omega_k\omega_{2k}=ij\}.
\]
We also denote
\[
N_{i, {\rm odd}}(\omega_m^n)=\sharp \{m\leq k \leq n: k \ \text{odd}, \ \omega_k=i\}.
\]

Then we have 
\[
{\mu(C_{2n}(\omega)) \over \mu(C_{n}(\omega))} = p_0^{N_{0, {\rm odd}}}p_1^{N_{1, {\rm odd}}}p_{00}^{N_{00}}p_{01}^{N_{01}}p_{10}^{N_{10}}p_{11}^{N_{11}},
\]
with $N_{i, {\rm odd}}=N_{i, {\rm odd}}(\omega_{n+1}^{2n})$, and $N_{ij}=N_{ij}(\omega_{n/2+1}^n)$.
Thus
\begin{align*}
\frac 2 n h_n(\omega) = &{N_{0, {\rm odd}} \over n/2} \log p_0+ {N_{1, {\rm odd}} \over n/2} \log p_1   \\
&+ {N_{00} \over n/2} \log p_{01} +  {N_{01} \over n/2} \log p_{01} + {N_{10} \over n/2} \log p_{10} +  {N_{11} \over n/2} \log p_{11} .
\end{align*}
By the classical strong Law of Large Numbers, for $\mu$-almost all $\omega$, for large enough $n$, 
\begin{align*}
 {N_{0, {\rm odd}} \over n/2} = p_0 + o(1),  \quad  {N_{1, {\rm odd}} \over n/2} = p_1 + o(1).
 \end{align*}
By Lemma \ref{lem:wk-w2k},  for $\mu$-almost all $\omega$, for large enough $n$, 
\begin{align*}
 {N_{00} \over n/2} =(1-x_n(\omega))\cdot p_{00} + o(1),  \quad {N_{01} \over n/2} =(1-x_n(\omega))\cdot p_{01} + o(1), 
 \end{align*}
 and
\begin{align*}
  {N_{10} \over n/2} =x_n(\omega)\cdot p_{10} + o(1),  \quad  {N_{11} \over n/2} =x_n(\omega)\cdot p_{11} + o(1).
\end{align*}
Therefore,
\begin{align*}
\frac 2 n h_n(\omega) = &p_0\log p_0+p_1\log p_1+ (1-x_n(\omega))( p_{00} \log  p_{00}+ p_{01} \log  p_{01}) \\
&+ x_n(\omega)( p_{10} \log  p_{10}+ p_{11} \log  p_{11})+ o(1).
\end{align*}

\end{proof}

By Lemma \ref{lem:h-n} and Corollary \ref{cor:lim-x-n}, we immediately have the following corollary.
\begin{cor}\label{cor:lim-hn}
For $\mu$-almost all  sequence $\omega$,
\begin{align*}
\lim_{n\to\infty}\frac 2 n h_n(\omega) = &p_0\log p_0+p_1\log p_1+ (1-\xi)( p_{00} \log p_{00}+ p_{01} \log  p_{01}) \\
&+ \xi( p_{10} \log  p_{10}+ p_{11} \log  p_{11}).
\end{align*}

\end{cor}

We close this section by the following corollary which gives the local entropy $h_\mu(\omega)$ of the measure $\mu$ for generic sequence $\omega$.
\begin{prop}\label{prop:local-ent}
For $\mu$-almost all sequence $\omega$,
\begin{align*}
h_\mu(\omega)
=& -{1 \over 2} \Big(p_0\log p_0+p_1\log p_1+ (1-\xi)( p_{00} \log  p_{00}+ p_{01} \log p_{01}) \\
&+ \xi( p_{10} \log  p_{10}+ p_{11} \log  p_{11})\Big).
\end{align*}
\end{prop}
\begin{proof}
By Corollary  \ref{cor:lim-hn}, we need only to show that for all large enough $n$,
\[
{\log \mu(C_n(\omega)) \over n }={h_n(\omega) \over n} +o(1).
\]
In fact, for all $k, n\in \N$
\[
{1 \over 2^kn} \log \mu(C_{2^kn}(\omega))= {1 \over 2^kn} \left(\sum_{i=1}^{k-1} h_{2^in}(\omega)+ \log \mu(C_n(\omega)) \right).
\]
Then $\mu$-almost surely, for all $n\in \N$, for large $k$, 
\[
{1 \over 2^kn}  \log \mu(C_{2^kn}(\omega))=  {h_n(\omega) \over n}+o(1).
\]
Applying Lemma \ref{lem:PS}, we complete the proof. 
\end{proof}


\bigskip

\section{Proof of Theorem \ref{thm:norm}}\label{sec:3}

We first prove the lower bound. We will need the measures defined in Subsection \ref{subsec:measures} and we will conclude by applying Billingsley's lemma for Bowen's entropy: Theorem \ref{lem:Bi}.

Given $\alpha\in [0,1]$, let $\mu_\alpha$ be the probability measure on $\Sigma$ constructed in Subsection \ref{subsec:measures}, by using the data
 $(p_0, p_1):=(1/2, 1/2)$ and 
\[
\begin{pmatrix}
   p_{00} & p_{01} \\
   p_{10} & p_{11} 
\end{pmatrix}
:=\begin{pmatrix}
   2\alpha & 1-2 \alpha \\
  1-2 \alpha & 2 \alpha 
\end{pmatrix}.
\]
Then
\[
\mu_\alpha(C_n(\omega_1\cdots\omega_{n})={1 \over 2^{\lceil n/2\rceil}} \cdot  \prod_{k=1}^{\lfloor n/2\rfloor} p_{\omega_{k}\omega_{2k}}.
\]

\smallskip
We will prove that the measure $\mu_\alpha$ is supported on the set $\mathcal{N}\cap A_\alpha$. 
\begin{lem} \label{typical}
We have
\[
\mu_{\alpha}(\mathcal{N} \cap A_\alpha) =1.
\]
\end{lem}

\begin{proof}
Note that by our choice of data,
\[
\xi= {p_1+p_{01} \over 2-p_{11}+p_{01}}= {{1\over 2}+1-2\alpha \over 2-2\alpha+1-2\alpha}={1 \over 2}.
\]
Hence by Proposition \ref{prop:wk-w2k}, for $\mu_{\alpha}$-almost all sequence $\omega$,
\[
\lim_{n\to\infty}\frac 1 n \sum_{k=1}^n \omega_k \omega_{2k} = \xi \cdot p_{11}={1 \over 2} \cdot 2\alpha=\alpha.
\]
Thus $\mu_\alpha(A_\alpha)=1$.

Now, we show $\mu_\alpha(\mathcal{N})=1$. 
We can divide the set of natural numbers into infinitely many subsets of the form 
$$A_k=\{2k-1, 4k-2, \ldots, 2^\ell (2k-1), \ldots\} \quad (k\geq 1).$$
 Let ${\cal B}_k$ be the $\sigma$-field generated by the events $\{\omega_{2^\ell(2k-1)}=1\}$, $\ell\geq 0$. Observe that for the measure $\mu_\alpha$ the $\sigma$-fields ${\cal B}_k$ are independent.
Observe further that $\mu_\alpha(\omega_{2^\ell(2k-1)}=1)=1/2$ for every $k, \ell$. Indeed, for $\ell=0$, it follows from the definition of $\mu_\alpha$. For the general $\ell$, it is proved by induction:
\begin{align*}
&\mu_\alpha(\omega_{2^\ell(2k-1)}=1)\\
=& \mu_\alpha(\omega_{2^\ell(2k-1)}=1 \wedge\omega_{2^{\ell-1}(2k-1)}=1) + \mu_\alpha(\omega_{2^\ell(2k-1)}=1 \wedge \omega_{2^{\ell-1}(2k-1)}=0) \\
=& \mu_\alpha(\omega_{2^{\ell-1}(2k-1)}=1)\cdot\mu_\alpha(\omega_{2^\ell(2k-1)}=1|\omega_{2^{\ell-1}(2k-1)}=1)\\
&+ \mu_\alpha(\omega_{2^{\ell-1}(2k-1)}=0) \cdot  \mu_\alpha(\omega_{2^\ell(2k-1)}=1 |\omega_{2^{\ell-1}(2k-1)}=0)\\
=&  {1\over 2} \cdot 2\alpha +  {1\over 2} \cdot  (1-2\alpha)  =  {1\over 2}.
\end{align*}

Consider now, for any $n \geq 1$, the word $\omega_{m+1}\cdots\omega_{m+n}$. If $m\geq n$ then the positions $m+1,\ldots,m+n$ come all from different $A_k$'s. Thus $\omega_{m+1},\ldots, \omega_{m+n}$ are independent and each of them takes values $0,1$ with probability $1/2$ respectively. That is, the measure $\mu$ restricted to such subset of positions is $(\frac 12, \frac 12)$-Bernoulli, and for any word $\eta\in \{0,1\}^n$ with $n\leq m$, the probability that we have $\omega_{m+i}=\eta_i$ for $i=1,\ldots,n$ equals $2^{-n}$. Thus, for a given word $\eta\in \{0,1\}^n$ we can divide $\N$ into intervals $[2^j+1, 2^{j+1}]$, inside all except initial finitely many of them (with $j< \log_2 n$), for any $\mu_\alpha$-generic sequence $\omega$, the frequency of appearance of $\eta$ equals $2^{-n}+O(2^{-j/2}j\log j)$, and this means that the $\mu_\alpha$-generic sequence $\omega$ is normal.
\end{proof}

Next, we will calculate the local entropy $h_{\mu_\alpha}(\omega)$ of the measure $\mu_\alpha$ for generic sequence $\omega$.
We denote for $t\in[0,1]$,
\[
H(t) = -t\log t -(1-t)\log (1-t),
\]
with convention $H(0)=H(1)=0$.

\begin{lem} \label{entropy-mu-alpha}
We have
\[
h_{\mu_\alpha}(\omega) = \frac 1 2 \log 2+ {H(2\alpha)}, \qquad \mu_{\alpha}-a.e.
\]
\end{lem}

\begin{proof}
By Proposition \ref{prop:local-ent}, we have for $\mu_\alpha$-almost all sequence $\omega$,
\begin{align*}
h_{\mu_\alpha}(\omega)
=& -{1 \over 2} \Big(p_0\log p_0+p_1\log p_1+ (1-\xi)( p_{00} \log  p_{00}+ p_{01} \log  p_{01}) \\
&+ \xi( p_{10} \log  p_{10}+ p_{11} \log  p_{11})\Big)\\
=& -{1 \over 2} \Big({1 \over 2}\log {1 \over 2}+{1 \over 2}\log {1 \over 2}+ (1-{1 \over 2})\big( 2\alpha \log (2\alpha)+ (1-2\alpha) \log (1-2\alpha)\big) \\
&+ {1 \over 2}\big( (1-2\alpha) \log (1-2\alpha)+ 2\alpha \log (2\alpha)\big)\Big)\\
=& \frac 1 2 \log 2+ {H(2\alpha)}.
\end{align*}

\end{proof}

Applying Theorem \ref{lem:Bi}, by Lemmas \ref{typical} and \ref{entropy-mu-alpha}, we immediately obtain 
\[
h_{\rm top} (\mathcal{N} \cap A_{\alpha}) \geq {1 \over 2} \log 2+ {1 \over 2} H(2\alpha).
\] 

To finish the proof of the lower bound we note that $A\subset A_0$ but the measure $\mu_0$ is actually supported on $A$,  that the measure $\mu_{1/2}$ is supported on $B$, and that the relation ${\mathcal N}\cap B \subset A_{1/2}$ follows from 
\[ 
\frac 1n\sharp\{n+1\leq j \leq 2n:  \omega_j = \omega_{2j} =1\} =\frac 1n\sharp\{n+1\leq j \leq 2n: \omega_j =1\}\to \frac 12
\]
being satisfied for every $\omega\in {\mathcal N}\cap B$.

\medskip
 For the upper bound, let us first observe that 
\[
\frac 1n \sum_{k=1}^n \omega_k \omega_{2k} \leq \frac 1n \sum_{k=1}^n \omega_k
\]
and the right hand side converges to $1/2$ for every normal sequence $\omega$. Thus, the set $\mathcal{N} \cap A_{\alpha}$ is empty for all $\alpha >1/2$.

To continue with the case $\alpha \leq 1/2$, we will need the following lemma.
\begin{lem} \label{lem:subseq}
Let $\omega$ be a normal sequence and let $(n_k = \ell_1 + k \ell_2)$ be an arithmetic subsequence of $\N$. Then $\omega$ restricted to the positions $(n_k)$ is normal.
\end{lem}
\begin{proof}
The result is originally due to Wall \cite{W49}. See also Kamae \cite{K73}.
\end{proof}

Let us fix some $m>0$. For  $N>m$ and $i=0,1,\ldots,m$, denote by $R(N,i)$ the set $\{2^i(2k-1), k\leq 2^{N-i-1}\}$ (for example, $R(N,0)$ is the set of odd numbers smaller than $2^N$). Further, let $$R(N,i,I):=R(N-2,i),$$ $$R(N,i,II):=R(N-1,i)\setminus R(N-2,i),$$ and $$R(N,i,III):=R(N,i)\setminus R(N-1,i).$$ Note here the following obvious relations
\begin{align*}
& 2 R(N,i,I)  = R(N, i+1, I) \cup R(N,i+1, II),\\
& 2R(N,i,II) = R(N,i+1,III),\\
& 2R(N,i,III) \cap R(N, i+1)  =\emptyset.
\end{align*}

We denote by $\mathcal{N}(N,m,\varepsilon)$ the set of sequences $\omega$ such that for all $n\geq N$ in each $R(n,i,\ast)$, $i=0,\ldots,m$, $\ast\in\{I, II, III\}$ the frequency of 1's is between $1/2-\varepsilon$ and $1/2+\varepsilon$. By Lemma \ref{lem:subseq}, 
\[
{\cal N} \subset \bigcap_{\varepsilon>0} \bigcap_{m=1}^\infty \bigcup_{N=m+1}^\infty \mathcal{N}(N,m,\varepsilon).
\]

Similarly, let us denote by $A(\alpha, N, \varepsilon)$ the set of sequences $\omega$ such that for all $n\geq N$ we have
\[
\alpha - \varepsilon <2^{-n+1} \sum_{j=1}^{2^{n-1}} \omega_j \omega_{2j} < \alpha+\varepsilon.
\]
Then
\[
A_\alpha = \bigcap_{\varepsilon>0} \bigcup_{N=1}^\infty A(\alpha, N,\varepsilon).
\]

To obtain the upper bound, we will estimate from above the number of cylinders $C_{2^N}(\omega_1,\ldots,\omega_{2^N})$ needed to cover the set $\mathcal{N}(N,m,\varepsilon)\cap A(\alpha,N,\varepsilon)$.
Let us fix $N, m, \varepsilon$. To find the cylinders $C_{2^N}(\omega_1,\ldots,\omega_{2^N})$, we should determine for each position $1\leq n \leq 2^N$, which value ($0$ or $1$), $\omega_n$ will take. To this end, it would be convenient that we partition the positions from $1$ to $2^N$ by several classes according to the values of $\omega_n$ and that of the couple $(\omega_n, \omega_{2n})$. We will introduce the following notations concerning the number of positions in each such class.
For $i=1,\ldots,m$, $k_1, k_2\in \{0,1\}$, and $\ast\in \{I, II\}$, we denote
\[
X_{k_1k_2, \ast}^i(\omega) := \sharp \{n\in R(N,i-1,\ast): \omega_n = k_1, \omega_{2n} = k_2\}.
\]
For example, $X_{01,I}^1(\omega)$ denotes the number of odd positions $n$ smaller than $2^{N-2}$ such that $\omega_n=0, \omega_{2n}=1$. Similarly, let
\[
X_{k_1,\ast}^i(\omega) := \sharp \{n\in R(N,i,\ast); \omega_n = k_1\}.
\]
The following relations are obvious: for any $i$,
\begin{align}
X_{10,I}^i + X_{11,I}^i = X_{1,I}^{i-1} \label{eq:1}
\\
X_{00,I}^i + X_{01,I}^i = X_{0,I}^{i-1} \label{eq:2}
\\
X_{10,II}^i + X_{11,II}^i = X_{1,II}^{i-1} \label{eq:3}
\\
X_{00,II}^i + X_{01,II}^i = X_{0,II}^{i-1} \label{eq:4}
\\
X_{01,I}^i + X_{11,I}^i = X_{1,I}^{i} + X_{1,II}^{i} \label{eq:5}
\\
X_{00,I}^i + X_{10,I}^i = X_{0,I}^{i} + X_{0,II}^{i} \label{eq:6}
\\
X_{01,II}^i + X_{11,II}^i = X_{1,III}^{i}  \label{eq:7} 
\\
X_{00,II}^i + X_{10,II}^i = X_{0,III}^{i}. \label{eq:8}
\end{align}

Note that for a sequence $\omega\in \mathcal{N}(N,m,\varepsilon)$ the right hand sides in all these relations are in range $2^{N-3-i}\cdot (1-\varepsilon, 1+\varepsilon)$. Thus by \eqref{eq:1}, \eqref{eq:2}, \eqref{eq:5}, \eqref{eq:6}, we have 
\begin{equation} \label{eqn:nnorm}
\Big|X_{11,I}^i-X_{00,I}^i \Big| \leq \varepsilon \cdot 2^{N-1-i},
\end{equation}
and by \eqref{eq:3}, \eqref{eq:4}, \eqref{eq:7}, \eqref{eq:8}, we have  
\begin{equation} \label{eqn:nnorm2}
\Big|X_{11,II}^i-X_{00,II}^i \Big| \leq \varepsilon \cdot 2^{N-1-i}.
\end{equation}

Note that once we know the values of the sequence
\begin{equation} \label{sequence-X}
(X_{00,I}^i, X_{11,I}^i, X_{00,II}^i, X_{11,II}^i)_{i=1}^m,
\end{equation}
then by \eqref{eq:1}-\eqref{eq:4}, the values of $(X_{0,I}^i, X_{1,I}^i, X_{0,II}^i, X_{1,II}^i)_{i=1}^{m-1}$ are also determined. 

Let us start the counting of the possible $(\omega_1, \dots, \omega_{2^N})$. The idea is as follows: we will first describe the sequences that can appear in $\mathcal{N}(N,m,\varepsilon)$, starting with what can happen on the odd positions $2k-1$, after that what can happen on positions of the form $2(2k-1)$ provided that the odd positions are already decided, and so on. Finally we will go back and add another condition, for our sequences to belong to $A(\alpha,N,\varepsilon)$.

Now assume that we know the values of the sequence \eqref{sequence-X} of number of positions. We will count how many possible $(\omega_1, \dots, \omega_{2^N})$ we can have, based on the information of the  values of \eqref{sequence-X}. 

The values of $\{\omega_n: n\in R(N,0)\}$ can be chosen in no more than $2^{2^{N-1}}$ ways. 
After we have chosen $\{\omega_n:  n\in R(N,i-1)\}$, we can choose $\{\omega_n: n\in R(N,i)\}$ in no more than
\[
{X_{1,I}^{i-1} \choose X_{11,I}^i} \cdot {X_{0,I}^{i-1} \choose X_{00,I}^i} \cdot {X_{1,II}^{i-1} \choose X_{11,II}^i} \cdot {X_{0,II}^{i-1} \choose X_{00,II}^i} 
\]
ways. Finally, after we have chosen $\{\omega_n: n\in R(N,i)\}$ for all $i\leq m$, we will still have $2^{N-m-1}$ positions left, which we can choose in no more than $2^{2^{N-m-1}}$ ways. 

To continue our estimation, we will use the following fact: for $k,n\in \mathbb{N}, 0\leq k\leq n$, 
\begin{align}\label{eq:approx}
\begin{split}
\log {n\choose k} &= n \left(- \frac kn \log \frac kn - \frac {n-k}n \log \frac {n-k}n\right) + O(\log n)\\
&= n\cdot H\big(\frac kn\big)+O(\log n),
\end{split}
\end{align}
where we recall that $H(t)=-t\log t-(1-t)\log(1-t)$ is a concave analytic function on $[0,1]$.

Note that $X_{1,I}^{i-1}$ and $X_{1,I}^{i-1}$ are in the range $2^{N-3-i}\cdot (1-\varepsilon, 1+\varepsilon)$. Applying \eqref{eq:approx}, by \eqref{eqn:nnorm}, we have 
\[
\log {X_{1,I}^{i-1} \choose X_{11,I}^i} +\log {X_{0,I}^{i-1} \choose X_{00,I}^i} = 2 \log{2^{N-3-i} \choose X_{11,I}^i} + 2^{N-3-i}O(\varepsilon).
\]
Similarly, by \eqref{eqn:nnorm2}, we have
\[
\log {X_{1,II}^{i-1} \choose X_{11,II}^i} +\log {X_{0,II}^{i-1} \choose X_{00,II}^i} = 2 \log{2^{N-3-i} \choose X_{11,II}^i} + 2^{N-3-i}O(\varepsilon).
\]

Thus for fixed values of the sequence \eqref{sequence-X}, the logarithm of  the total number $Z$ of cylinders needed to cover the corresponding part of $\mathcal{N}(N,m,\varepsilon)$ is not larger than
\begin{align*}
& \log Z\big((X_{00,I}^i, X_{11,I}^i, X_{00,II}^i, X_{11,II}^i)_{i=1}^m\big)\\
 \leq & (2^{N-1}+2^{N-m-1}) \log 2 \\
 & + \sum_{i=1}^m \left(2 \log{2^{N-3-i} \choose X_{11,I}^i} 
 + 2 \log{2^{N-3-i} \choose X_{11,II}^i} + 2^{N-3-i}O(\varepsilon)\right).
\end{align*}
By \eqref{eq:approx}, applying the Jensen inequality for the concave function $H$, we get
\begin{align*}
&
\log Z\big((X_{00,I}^i, X_{11,I}^i, X_{00,II}^i, X_{11,II}^i)_{i=1}^m\big)\\
 \leq & (2^{N-1}+2^{N-m-1}) \log 2 \\
 &+ \sum_{i=1}^m 2^{N-i-1} \cdot H\left(\frac 1 {\sum_{i=1}^m 2^{N-i-1}} \cdot \sum_{i=1}^m 2^{N-i-2} \frac {X_{11,I}^i+X_{11,II}^i}{2^{N-i-3}}\right) \\
 &+ \sum_{i=1}^m 2^{N-i-3} \cdot O(\varepsilon).
\end{align*}
Hence,
\begin{align}\label{estimate-logZ}
\begin{split}
&
\log Z\big((X_{00,I}^i, X_{11,I}^i, X_{00,II}^i, X_{11,II}^i)_{i=1}^m\big)\\
 \leq & 2^{N-1} \log 2 + 2^{N-1} H\left(2^{-N+2} \sum_{i=1}^m (X_{11,I}^i+X_{11,II}^i)\right) \\
 &+ 2^N \cdot (O(\varepsilon + 2^{-m})).
 \end{split}
\end{align}

Moreover, there are no more than 
\begin{align}\label{eq:nb-choice}
\prod_{i=1}^m 2^{4(N-i-3)} < 2^{4mN}=o(2^{2^N})
\end{align} 
possible values of $(X_{00,I}^i, X_{11,I}^i, X_{00,II}^i, X_{11,II}^i)_{i=1}^m$. We remark here that the number of possibilities is much less, because of \eqref{eqn:nnorm} and \eqref{eqn:nnorm2}. But the estimate \eqref{eq:nb-choice} is enough for us. 


On the other hand, for all $\omega\in A(\alpha,N,\varepsilon)$,
\[
\left|2^{-N+1} \sum_{i=1}^m (X_{11,I}^i+X_{11,II}^i) - \alpha\right| < \varepsilon + 2^{-m}.
\]
Thus by the estimates \eqref{estimate-logZ} and \eqref{eq:nb-choice}, the logarithm of the number of cylinders $C_{2^N}(\omega_1,\ldots,\omega_{2^N})$ needed to cover the set $\mathcal{N}(N,m,\varepsilon)\cap A(\alpha,N,\varepsilon)$ is less than
\begin{align*}
&o({2^N}) +\log Z\big((X_{00,I}^i, X_{11,I}^i, X_{00,II}^i, X_{11,II}^i)_{i=1}^m\big)\\
\leq & 2^{N-1} \log 2 + 2^{N-1} \left(H(2\alpha) +O(\varepsilon + 2^{-m})\right)+ 2^N \cdot (O(\varepsilon + 2^{-m})).
\end{align*}
Dividing the above value by $2^N$, and passing with $m,N$ to infinity and with $\varepsilon$ to $0$, we finish the proof of the upper bound.

\section{Proofs of Theorem \ref{thm:freq} and Corollary \ref{cor}}\label{sec:4}


Given $p,q\in [0,1]$, let $\mu_{p,q}$ be the probability measure on $\Sigma$ constructed in Subsection \ref{subsec:measures}, by using the data
$(p_0, p_1):=(1-p, p)$ and 
\[
\begin{pmatrix}
   p_{00} & p_{01} \\
   p_{10} & p_{11} 
\end{pmatrix}
:=\begin{pmatrix}
   1-p & p \\
  1-q  & q 
\end{pmatrix}.
\]
%
%
%
%

The proof of Theorem \ref{thm:freq} is based on the following lemmas.
\begin{lem} \label{typical-2}
If $p=(2\theta-\alpha)/(2-\theta)$ and $q=\alpha / \theta$, then
\[
\mu_{p,q}(E_\theta\cap A_\alpha) =1.
\]
\end{lem}

\begin{proof}

By Proposition \ref{prop:lim-wk}, $\mu_{p,q}$-almost surely
\begin{equation} \label{eq:freq}
\lim_{n\to\infty} {1 \over n}\sum_{k=1}^n\omega_k =\xi= {p_1+p_{01} \over 2-p_{11}+p_{01}}=  \frac {2p} {2+p-q}=\theta,
\end{equation}
where the last equality comes from  the choices of $p$ and $q$.
Thus $\mu_{p,q}(E_\theta)=1$.

On the other hand, by Proposition \ref{prop:wk-w2k}, for $\mu_{p,q}$-almost all sequence $\omega$,
\[
\lim_{n\to\infty}\frac 1 n \sum_{k=1}^n \omega_k \omega_{2k} = \xi \cdot p_{11} = \theta \cdot q= \alpha.
\]
By Lemma \ref{lem:PS}, we conclude $\mu_{p,q}(A_\alpha)=1$.

\end{proof}

\begin{lem} \label{entropy-measure-2}
For  $p=(2\theta-\alpha)/(2-\theta)$ and $q=\alpha / \theta$, we have
\[
h_{\mu_{p,q}}(\omega)=\big(1-{\theta \over 2}\big)H\Big({2\theta- \alpha \over 2-\theta}\Big) +{\theta \over 2}H\Big({\theta- \alpha \over \theta}\Big), \qquad \mu_{p,q}-a.e.
\]
\end{lem}

\begin{proof}
By Proposition \ref{prop:local-ent}, we have for $\mu_{p,q}$ almost all $\omega\in \Sigma$,
\begin{align*}
h_{\mu_{p,q}}(\omega)= 
& -{1 \over 2} \Big( (1-p) \log (1-p)+p \log p+(1-\theta)(1-p) \log (1-p) \\
&+(1-\theta)p \log p +\theta(1-q) \log (1-q)+\theta q \log q \Big)\\
=& {1 \over 2}   \Big((2-\theta)H(p)+ \theta H(q)\Big)\\
=& \big(1-{\theta \over 2}\big)H\Big({2\theta- \alpha \over 2-\theta}\Big) +{\theta \over 2}H\Big({\theta- \alpha \over \theta}\Big).
\end{align*}
\end{proof}

\begin{lem} \label{entropy-measure-all}
If $\theta\notin [\alpha, (2+\alpha)/3]$ we have $E_\theta\cap A_\alpha=\emptyset$, otherwise
for  $p=(2\theta-\alpha)/(2-\theta)$ and $q=\alpha / \theta$, we have for all $\omega\in E_\theta\cap A_\alpha$,
\[
h_{\mu_{p,q}}(\omega)=\big(1-{\theta \over 2}\big)H\Big({2\theta- \alpha \over 2-\theta}\Big) +{\theta \over 2}H\Big({\theta- \alpha \over \theta}\Big).
\]

\end{lem}

\begin{proof}
Observe that for any $\omega\in E_\theta\cap A_\alpha$, for any small $\epsilon>0$, for $n$ large enough, we have
\begin{align*}
N_1(\omega_{n/2}^n) &\in \Big[{\theta n \over 2}(1-\epsilon),  \ {\theta n \over 2}(1+\epsilon)\Big]\\
N_1(\omega_n^{2n}) &\in \Big[{\theta n}(1-\epsilon), \ {\theta n}(1+\epsilon)\Big]\\
N_{11}(\omega_n^{2n}) &\in \Big[{\alpha n \over 2}(1-\epsilon), \ {\alpha n \over 2}(1+\epsilon)\Big].
\end{align*}
The obvious inequalities $$N_{11}(\omega_n^{2n}) \leq N_1(\omega_{n/2}^n) $$ and $$N_1(\omega_n^{2n})-N_{11}(\omega_n^{2n}) \leq n/2 + N_0(\omega_{n/2}^n) =n-N_1(\omega_{n/2}^n)$$ imply $\theta\in [\alpha, (2+\alpha)/3]$. Furthermore, we have
\begin{align*}
&\log \mu_{p,q}(C_{2n}(\omega)) - \log \mu_{p,q}(C_n(\omega))\\
=& {N_{11}(\omega_n^{2n})}\log q +\big({N_1(\omega_{n/2}^n)-N_{11}(\omega_n^{2n})}\big)\log (1-q)\\
&+\big({N_1(\omega_n^{2n})-N_{11}(\omega_n^{2n})}\big)\log p \\
&+\big({n-N_1(\omega_{n/2}^n)-N_1(\omega_n^{2n})+N_{11}(\omega_n^{2n})}\big)\log (1-p)\\
=& n\left( \big(1-{\theta \over 2}\big)H\Big({2\theta- \alpha \over 2-\theta}\Big) +{\theta \over 2}H\Big({\theta- \alpha \over \theta}\Big) + \epsilon \cdot O(1)\right).
\end{align*}
Hence, by the same argument as in the proof of Proposition \ref{prop:local-ent}, 
we have
for all $\omega\in E_\theta\cap A_\alpha$,
\[
h_{\mu_{p,q}}(\omega)=\lim_{n\to\infty} {-\log \mu_{p,q}(C_n(\omega)) \over n }=\big(1-{\theta \over 2}\big)H\Big({2\theta- \alpha \over 2-\theta}\Big) +{\theta \over 2}H\Big({\theta- \alpha \over \theta}\Big).
\]

\end{proof}

\begin{proof}[Proof of Theorem \ref{thm:freq}]
We note that by Theorem \ref{lem:Bi}, the lower bound of Theorem \ref{thm:freq} follows from Lemmas \ref{typical-2} and \ref{entropy-measure-2} and the upper bound follows from Lemma \ref{entropy-measure-all}. We thus have completed our proof.
\end{proof}

\begin{proof}[Proof of Corollary \ref{cor}]

Note that for any $\theta$, $E_\theta \cap A_\alpha \subset  A_\alpha $. Thus, if $h_{\rm top}(E_\theta \cap A_\alpha)=h_{\rm top}(A_\alpha)$ for some $\theta$, then this $\theta$ is the maximal point for the entropy formula of $h_{\rm top}(E_\theta \cap A_\alpha)$. Remark that the entropy formula of $h_{\rm top}(E_\theta \cap A_\alpha)$ in Theorem \ref{thm:freq} is analytic and concave with respect to the variable $\theta$ and the partial derivative
\[
{\partial h_{\rm top}(E_\theta \cap A_\alpha) \over \partial \theta}= {1 \over 2} \log {\theta(2-3\theta+\alpha)^3 \over (2\theta-\alpha)^2(\theta-\alpha)(2-\theta) }=0
\]
if and only if (\ref{eq:max}) holds. We then have proved the first assertion.

By Theorems \ref{thm:KPS} and \ref{thm:PS},  $h_{\rm top}(A)=h_{\rm top}(A_0)$. 
By Theorem \ref{thm:freq}, $h_{\rm top}(E_\theta \cap A)=h_{\rm top}(E_\theta \cap A_0)$. Then the second assertion follows by taking $\alpha=0$.
\end{proof}

%
%

\section*{Acknowlegdements} We thank the referees for the valuable remarks and suggestions which significantly improves the presentation of the paper. We thank one of the referees  for pointing out us the reference \cite{W49} of Wall's Ph.D. thesis where we can find the original version of Lemma \ref{lem:subseq}.

\bibliographystyle{alpha}

\begin{thebibliography}{PSSS12}

\bibitem[ABC19]{ABC} C.  Aistleitner, V. Becher and O. Carton, 
 {\em Normal numbers with digit dependencies}, Trans. Amer. Math. Soc., {\bf 372} (2019), 4425--4446.
 
\bibitem[B61]{Bi61} P. Billingsley, {\em Hausdorff dimension in probability theory. II}, Illinois J.
Math., {\bf 5}, 291--298.

\bibitem[B73]{B73}
R. Bowen.
\newblock {\em Topological entropy for noncompact sets},
\newblock  Trans. Amer. Math. Soc., {\bf 184}:125--136, 1973.

\bibitem[F90]{F90} K. Falconer, {\it Fractal Geometry, Mathematical Foundations and Application}, Wiley, 1990.


\bibitem[FLM12]{FLM} A.H. Fan, L.M. Liao and J.H. Ma, {\em Level sets of multiple ergodic averages},
\newblock Monatsh. Math. \textbf{168} (2012), 17--26.





\bibitem[FSW16]{FSW16} A.H. Fan, J. Schmeling and M. Wu, {\em Multifractal analysis of some multiple ergodic averages}, Adv. Math. \textbf{295} (2016), 271--333.


\bibitem[K73]{K73} T. Kamae,  {\em Subsequences of normal sequences}, Israel J. Math. {\bf 16} (1973), 121-149.

\bibitem[KPS12]{KPS12}
 R. Kenyon, Y. Peres and B. Solomyak,  {\em Hausdorff dimension for fractals invariant under the multiplicative integers},  Ergodic Theory and Dynamical Systems,
\textbf{32} (2012), no. 5,  1567--1584.

\bibitem[K12]{K12}
 Y. Kifer, {\em A nonconventional strong law of large numbers and fractal dimensions of some multiple recurrence sets}, Stoch. Dyn., {\bf 12} (3) (2012),
21p.

\bibitem[MW08]{MW08} J.H. Ma, Z.Y. Wen, {\em A Billingsley type theorem for Bowen entropy}, C. R. Math. Acad. Sci. Paris {\bf 346} (2008), no. 9-10, 503--507.

\bibitem[PS12]{PeSo12}
 Y. Peres and B. Solomyak,  {\em Dimension spectrum for a nonconventional ergodic average},  Real Analysis Exchange, \textbf{37} (2011), no. 2, 375-388. 

\bibitem[W49]{W49} D. D. Wall, Normal Numbers. Ph.D. thesis,
University of California, Berkeley, California, USA, 1949.

\end{thebibliography}

\end{document}